\documentclass[12 pt]{article}
\usepackage{amssymb}
\usepackage{amsmath}
\usepackage{amsthm}
\allowdisplaybreaks
\usepackage{amscd}
\usepackage{graphicx}
\usepackage[all,cmtip]{xy}

\usepackage{geometry}

\usepackage{ifpdf}
\ifpdf
 \usepackage[colorlinks=true,linkcolor=blue,final,backref=page,hyperindex]{hyperref}
\else
  \usepackage[colorlinks,final,backref=page,hyperindex,hypertex]{hyperref}
\fi

\def\vbar{\mathchoice{\vrule height6.3ptdepth-.5ptwidth.8pt\kern- .8pt}
{\vrule height6.3ptdepth-.5ptwidth.8pt\kern-.8pt} {\vrule
height4.1ptdepth-.35ptwidth.6pt\kern-.6pt} {\vrule
height3.1ptdepth-.25ptwidth.5pt\kern-.5pt}}

\def\<{\langle}
\def\>{\rangle}

\newtheorem{thm}{Theorem}[section]

\newtheorem{cor}[thm]{Corollary}
\newtheorem{prop}[thm]{Proposition}
\newtheorem{exam}[thm]{Example}
\newtheorem{nonexam}[thm]{Non-Example}
\newtheorem{rmk}[thm]{Remark}

\theoremstyle{definition}
\newtheorem{defn}{Definition}[section]

\theoremstyle{remark}
\begin{document}
\title{On compatible Leibniz algebras}
\author{
{ Abdenacer Makhlouf $^{1}$\footnote {  E-mail: abdenacer.makhlouf@uha.fr}
 \ and
Ripan Saha$^{2}$\footnote {Corresponding author,  E-mail:  ripanjumaths@gmail.com}
}\\
{\small 1.  Universit\'{e} de Haute Alsace, IRIMAS-D\'epartement de math\'{e}matiques},\\
{\small 18, rue des Fr\`{e}res Lumi\`{e}re 68093 Mulhouse, France} \\
{\small 2.  Department of Mathematics, Raiganj University
Raiganj 733134, West Bengal, India }}
\date{}
\maketitle
\begin{abstract}
In this paper, we study compatible Leibniz algebras. We characterize compatible Leibniz algebras in terms of Maurer-Cartan elements of a suitable differential graded Lie algebra. We define a cohomology theory of compatible Leibniz algebras which in particular controls a  one-parameter formal deformation theory of this algebraic structure. Motivated by a classical application of cohomology, we moreover study the abelian extension of compatible Leibniz algebras.
\end{abstract}




\numberwithin{equation}{section}

\section{Introduction}
In \cite{loday-di}, J.-L. Loday introduced some new type  of algebras  along with their (co)homologies and studied the associated operads. Leibniz algebras and their Koszul duals, Zinbiel algebras are examples of such algebras. A Leibniz algebra is a vector space $\mathfrak g$ equipped with a bilinear map $[~,~]$ satifying the following Leibniz identity: 
$$[x,[y,z]]= [[x,y],z]-[[x,z],y] ~~\mbox{for}~x,~y,~z \in \mathfrak g.$$
In the presence of skew-symmetry the Leibniz identity reduces
to Jacobi identity, and therefore, Lie algebras are examples of
Leibniz algebras. Hence, Leibniz algebras present a  non-antisymmetric analogue of Lie algebras. In fact, such algebras had been first considered by Bloch \cite{bloch} in 1965, who called them $D$-algebras. Loday \cite{loday-book} has investigated Leibniz algebras in connection with properties of cyclic homology and Hochschild homology of matrix algebras. Leibniz algebras also appeared in Mathematical Physics and in the literature they are also known Loday algebras.

During the last decade Leibniz algebras and their properties
have been investigated intensively, however, there are various aspects where
information about these algebras are not known. In this paper, we introduce and study a notion of compatible Leibniz algebras. Two Leibniz algebras $(\mathfrak g, [~,~]_1)$ and $(\mathfrak g, [~,~]_2)$ over a field $\mathbb{K}$ are called compatible if for any $\lambda_1, \lambda_2 \in \mathbb{K}$, the following bilinear operation
$$[x,y]= \lambda_1[x,y]_1 + \lambda_2 [x,y]_2,$$
for all $x,y \in \mathfrak g$ defines a Leibniz algebra structure on $\mathfrak g$. In fact, any linear combination of the brackets defines a Leibniz algebra is equivalent to the sum of brackets $[~,~]_1+[~,~]_2$ defines a Leibniz algebra structure on $\mathfrak{g}$. Golubchik and Sokolov \cite{golu1, golu2, golu3} studied compatible Lie algebras and showed that compatible Lie algebras are closely related to Nijenhuis deformations of Lie algebras, classical Yang-Baxter equations and  principal chiral fields. Odesskii and Sokolov \cite{Odesskii2, Odesskii3} studied compatible associative algebras and their relations with associative Yang-Baxter equations, quiver representations and also studied compatible Lie brackets related to elliptic curves \cite{Odesskii1}. Compatible bialgebras were discussed in \cite{mar}. In the geometric context, compatible Poisson structures appeared in the mathematical study of biHamiltonian mechanics \cite{mag-mor, kos, dotsenko}. In \cite{CDM}, the authors studied the compatible associative algebras from the cohomological point of view, and in a similar context  compatible Hom-associative algebras  were considered in \cite{CS22}. In \cite{comp-lie}, the authors studied compatible Lie and Hom-Lie algebras and also studied the cohomology and deformation theory for those algebras. Homotopy version of compatible Lie algebras were studied in \cite{das22}. For some more interesting study of various type of compatible algebras and their applications, see \cite{wu, uchino, stro}.

The  algebraic deformation theory for associative algebras based on formal power series were introduced by Gerstenhaber in \cite{gers, gers2}, where it was shown that they are intimately connected to cohomology groups. Nijenhuis and Richardson extended one-parameter formal  deformation theory to  Lie algebras in \cite{nij-ric}. Later following Gerstenhaber, deformation theory are studied extensively for other algebraic structures. To study deformation theory of a type of algebra one needs a suitable cohomology, called deformation cohomology which controls deformations in question. In the case of associative algebras, Gerstenhaber showed that deformation cohomology is Hochschild cohomology \cite{hoch} and for Lie algebras, the associated deformation cohomology is Chevalley-Eilenberg cohomology.

The work of the present paper is organised as follows: In Section \ref{sec2}, we recall the definition, examples, representation, and cohomology of Leibniz algebras. In Section \ref{sec3}, we define the notion of compatible Leibniz algebras, give some examples as well as a classification in low dimensions. Moreover, we discuss representations of compatible Leibniz algebras. In Section \ref{sec5}, we construct a suitable differential graded Lie algebra and characterize the compatible Leibniz algebras as a Maurer-Cartan elements of this graded Lie algebra. In Section \ref{sec6}, we define a cohomology theory for compatible Leibniz algebras by combining both the cochains for the given Leibniz algebras. In Section \ref{sec7}, we define one-parameter formal deformation theory for compatible Leibniz algebras, study infinitesimal deformations, and show that our cohomology defined in Section \ref{sec6} is the deformation cohomology. Finally, in Section \ref{sec8}, we study abelian extensions for compatible Leibniz algebras and show that equivalence classes of such extensions are in one-to-one correspondence with the elements of a second cohomology group.
\section{Preliminaries}\label{sec2}
In this section, we recall the basics of Leibniz algebras which will be required throughout the paper.
\begin{defn}   
Let $\mathbb{K}$ be a field. A Leibniz algebra is a vector space $\mathfrak g$ over $\mathbb{K},$ equipped with a $\mathbb{K}$-bilinear map (known as bracket operation) that  satisfies the Leibniz identity: 
$$[x,[y,z]]= [[x,y],z]-[[x,z],y] ~~\mbox{for all }~x,~y,~z \in \mathfrak g.$$
\end{defn}
Any Lie algebra is automatically a Leibniz algebra, as in the presence of skew symmetry, the Jacobi identity is equivalent to the Leibniz identity. Therefore, Leibniz algebras are generalization of Lie algebras.
\begin{exam}\label{example-1}
Suppose $(\mathfrak g,d)$ is a differential Lie algebra with the Lie bracket $[~,~]$. Then $\mathfrak g$ inherits a Leibniz algebra structure with the bracket operation $[x,y]_d:= [x,dy]$. This new bracket on $\mathfrak g$ is  called the derived bracket.
\end{exam}
\begin{exam}\label{ex}\label{example-2}
Suppose $\mathfrak g$ is a three dimensional vector space spanned by \linebreak$\{e_1,~e_2,~e_3\}$ over $\mathbb{C}$. Define a bilinear map $[~,~]: \mathfrak g\times \mathfrak g \longrightarrow \mathfrak g$ by $[e_1,e_3]=e_2$ and $[e_3,e_3]= e_1$, all other products of basis elements being $0$. Then $(\mathfrak g,[~,~])$ is a 3-dimensional Leibniz algebra over $\mathbb{C}$ \cite{A3}. 
\end{exam}

\begin{defn}
A morphism $\phi : (\mathfrak g_1, [~,~]_1) \rightarrow (\mathfrak g_2, [~,~]_2)$ of Leibniz algebras is a $\mathbb{K}$-linear map satisfying
$$\phi ([x,y]_1) = [\phi (x), \phi (y)]_2,~~\forall x, y \in \mathfrak g_1.$$
\end{defn}

Let $\mathfrak g$ be a Leibniz algebra. A representation of $\mathfrak g$ is a vector space $M$ equipped with two actions (left and right) of $\mathfrak g$,
$$[~,~]:\mathfrak g\times M\longrightarrow M~~\mbox{and}~[~,~]:M \times \mathfrak g \longrightarrow M ~~\mbox{such that}~$$ 
$$[x,[y,z]]=[[x,y],z]-[[x,z],y]$$
holds, whenever one of the variables is from $M$ and the two others from $\mathfrak g$.

For all $n\geq 0$, set $CL^n({\mathfrak g}; {M}):= \mbox{Hom} _\mathbb{K}({\mathfrak g}^{\otimes n}, {M}).$ Define
$$\delta^n : CL^n({\mathfrak g}; {M})\longrightarrow CL^{n+1}(\mathfrak g; M)$$ 
as follows:
\begin{equation*}
\begin{split}
&\delta^nf(x_1,\ldots,x_{n+1})\\
&:= [x_1,f(x_2,\ldots,x_{n+1})] + \sum_{i=2}^{n+1}(-1)^i[f(x_1,\ldots,\hat{x}_i,\ldots,x_{n+1}),x_i]\\
&+ \sum_{1\leq i<j\leq n+1}(-1)^{j+1}f(x_1,\ldots,x_{i-1},[x_i,x_j],x_{i+1},\ldots,\hat{x}_j,\ldots, x_{n+1}).
\end{split}
\end{equation*}
Then $(CL^*(\mathfrak g; M), \delta)$ is a cochain complex, whose cohomology is called the cohomology of the Leibniz algebra $\mathfrak g$ with coefficients in the representation $M$. We denote the $n$th cohomology by $HL^n(\mathfrak g; M)$.  Any Leibniz algebra is a represenation over itself. The  $n$th cohomology of $\mathfrak g$ with coefficients in itself is denoted by $HL^n(\mathfrak g; \mathfrak g).$

A permutation $\sigma \in S_{n}$ is called a $(p,q)$-shuffle, if $ n=p+q,~\text{and}~\sigma(1)<\sigma(2)<\cdots<\sigma(p)$, and $\sigma(p+1)<\sigma(p+2)<\cdots<\sigma(p+q)$. We denote the set of all $(p,q)$-shuffles in $S_{p+q}$ by $Sh(p,q)$.

For $\alpha \in CL^{p+1}(\mathfrak g; \mathfrak g)$ and $\beta \in CL^{q+1}(\mathfrak g; \mathfrak g)$, define $\alpha \circ \beta \in CL^{p+q+1}(\mathfrak g; \mathfrak g)$ by
\begin{equation*}
\begin{split}
&\alpha \circ \beta (x_1,\ldots,x_{p+q+1} )\\
=&~\sum_{k=1}^{p+1}(-1)^{q(k-1)}\{\sum_{\sigma \in Sh(q,p-k+1)}sgn(\sigma)\alpha(x_1,\ldots,x_{k-1},\beta(x_k,x_{\sigma(k+1)},\ldots,x_{\sigma(k+q)}),\\
&~~~~~~~~~~~~~~~~~~~~~~~~~~~~~~~~~~~~~~~~~~~~~~~~~~~~~x_{\sigma(k+q+1)},\ldots,x_{\sigma(p+q+1)}) \}.
\end{split}
\end{equation*}

It is well-known \cite{bala} that the graded cochain module 
$CL^{*}(\mathfrak g; \mathfrak g)=\bigoplus_{p} CL^p(\mathfrak g; \mathfrak g)$ equipped with the following bracket operation  
$$[\alpha,\beta]=\alpha \circ \beta + (-1)^{pq+1} \beta \circ \alpha
~~\mbox{for}~ \alpha \in CL^{p+1}(\mathfrak g; \mathfrak g)~~\mbox{and}~\beta \in CL^{q+1}(\mathfrak g; \mathfrak g)$$
 and the differential map $d$  by $d \alpha =(-1)^{|\alpha|}\delta \alpha~\mbox{for}~\alpha \in CL^{*}(\mathfrak g; \mathfrak g) $ is a differential graded Lie algebra.

\section{Compatible Leibniz algebras}\label{sec3}

In this section, we define the notion of compatible Leibniz algebras. We discuss some examples and define a representation of such algebras.

\begin{defn}\label{defn cl}
Two Leibniz algebras $(\mathfrak g, [~,~]_1)$ and $(\mathfrak g, [~,~]_2)$ over a field $\mathbb{K}$ are called compatible if for any $\lambda_1, \lambda_2 \in \mathbb{K}$, the following bilinear operation
\begin{align}\label{equ1}
[x,y]= \lambda_1[x,y]_1 + \lambda_2 [x,y]_2,
\end{align}
for all $x,y \in \mathfrak g$ defines a Leibniz algebra structure on $\mathfrak g$.
\end{defn}
If $(\mathfrak g, [~,~]_1)$ and $(\mathfrak g, [~,~]_2)$ are compatible Leibniz algebras, then we denote it by $(\mathfrak g, [~,~]_1, [~,~]_2)$.

\begin{rmk}
The condition (\ref{equ1}) that the binary operation $[~,~]$ is a Leibniz bracket is equivalent to the following condition:
\begin{align}
[x,[y,z]_1]_2 + [x, [y,z]_2]_1=[[x,y]_1,z]_2 + [[x,y]_2,z]_1 - [[x,z]_1,y]_2- [[x,z]_2,y]_1.
\end{align}
\end{rmk}
In view of the above remark, we can restate the Definition \ref{defn cl} as follows:

A compatible Leibniz algebra is a triple $(\mathfrak g, [~,~]_1, [~,~]_2)$ such that 
\begin{itemize}
\item[i.] $(\mathfrak g, [~,~]_1)$ is a Leibniz algebra.
\item[ii.] $(\mathfrak g, [~,~]_2)$ is a Leibniz algebra.
\item[iii.] $[x,[y,z]_1]_2 + [x, [y,z]_2]_1=[[x,y]_1,z]_2 + [[x,y]_2,z]_1 - [[x,z]_1,y]_2- [[x,z]_2,y]_1,$ for all $x,y,z \in \mathfrak g.$
\end{itemize}

\begin{prop}
A pair $(m_1, m_2)$ of bilinear maps on a vector space $\mathfrak{g}$ defines a compatible Leibniz algebra structure on $\mathfrak{g}$ if and only if
\begin{align*}
[m_1, m_1] = 0, ~~~ [m_1, m_2] = 0,~~~~ \text{ and } ~~~~ [m_2, m_2] = 0 .
\end{align*}
\end{prop}
\begin{proof}
Using the Maurer-Cartan characterization, it is well-known that $m_1$ and $m_2$ defines Leibniz algebras on $\mathfrak{g}$ if and only if $[m_1,m_1]=0,~[m_2,m_2]=0$ respectively.
Note that 
$$(m_1\circ m_2)(x_1, x_2, x_3) = m_1(x_1, m_2(x_2, x_3))- m_1(m_2(x_1, x_2), x_3) + m_1(m_2(x_1, x_3), x_2).$$
This implies
\begin{align*}
&[m_1, m_2] \\
&= m_1\circ m_2 +m_2\circ m_1\\
&= m_1(x_1, m_2(x_2, x_3)) + m_2(x_1, m_1(x_2, x_3)) + m_1(m_2(x_1, x_3), x_2) + m_2(m_1(x_1, x_3), x_2) \\
&- m_1(m_2(x_1, x_2), x_3) - m_2(m_1(x_1, x_2), x_3).
\end{align*}
Therefore, $(\mathfrak{g}, m_1, m_2)$ is a compatible Leibniz algebra if and only if 
$$[m_1,m_1]=0,~[m_1, m_2]=0,~[m_2, m_2]=0.$$
\end{proof}

\begin{exam}
A Leibniz algebra $(\mathfrak{g}, [~,~])$ is called abelian if $[x,y]=0$, for all $x,y\in \mathfrak{g}$. Any Leibniz algebra $\mathfrak{g}$ is compatible with abelian Leibniz algebra.
\end{exam}

\begin{exam}
Let $\mathfrak{g}$ be a three dimensional vector space over $\mathbb{K}$ with a basis $\lbrace e_1, e_2, e_3 \rbrace$. Consider two Leibniz algebras $(\mathfrak{g}, [~,~]_1)$ and $(\mathfrak{g}, [~,~]_2)$ with non-zero brackets on the basis elements defined as
\begin{align*}
&[e_1,e_1]_1=e_3;\\
& [e_1,e_1]_2=e_2,~[e_2,e_1]_2=e_3.
\end{align*}
One can easily check that those two Leibniz algebras are compatible to each other.
\end{exam}

\begin{nonexam}
Let $\mathfrak{g}$ be a three dimensional vector space with a basis $\lbrace e_1, e_2, e_3 \rbrace$. Consider two Leibniz algebras $(\mathfrak{g}, [~,~]_1)$ and $(\mathfrak{g}, [~,~]_2)$ with non-zero brackets on the basis elements defined as
\begin{align*}
&[e_1,e_2]_1=e_3,~[e_2,e_1]_1=-e_3;\\
& [e_1,e_1]_2=e_2,~[e_2,e_1]_2=e_3.
\end{align*}
Now consider the bracket 
$$[x,y]= [x,y]_1 + [x,y]_2.$$
With respect to the above bracket, we have the following non-zero brackets on the basis elements
\begin{align*}
&[e_1,e_1]=e_2, ~[e_2,e_2]=e_3,~[e_1,e_2]=e_3.
\end{align*}
Observe that if $(\mathfrak{g}, [~,~])$ is a Leibniz algebra, then we have
$$[e_1,[e_1,e_1]]= [[e_1,e_1],e_1]-[[e_1,e_1],e_1].$$
This implies $e_3=0$, which is absurd.
\end{nonexam}

\begin{defn}
Let $(\mathfrak{g}, [~,~]_1, [~,~]_2)$ and $(\mathfrak{g}^{'}, [~,~]^{'}_1, [~,~]^{'}_2)$ be two compatible Leibniz algebras. A morphism between compatible Leibniz algebras $\mathfrak{g}$ and $\mathfrak{g}^{'}$ is a linear map $f: \mathfrak{g} \to \mathfrak{g}^{'}$ such that
\begin{align}
& f\circ [~,~]_1 = [~,~]^{'}_1 \circ (f\otimes f);\\
& f\circ [~,~]_2 = [~,~]^{'}_2 \circ (f\otimes f).
\end{align}
\end{defn}

\begin{defn}
Let $(\mathfrak{g}, [~,~]_1, [~,~]_2)$ be a compatible Leibniz algebra. A compatible $\mathfrak{g}$-bimodule is a quintuple $(M, l_1, r_1, l_2, r_2)$, where $M$ is a vector space and 
\begin{align}
\begin{cases}
l_1 : \mathfrak{g}\otimes M \to M,\\
r_1: M\otimes \mathfrak{g} \to M,
\end{cases}
;
\begin{cases}
l_2 : \mathfrak{g}\otimes M \to M,\\
r_2: M\otimes \mathfrak{g} \to M,
\end{cases}
\end{align}
are bilinear maps satisfying:
\begin{itemize}
\item[i.] $(M, l_1, r_1)$ is a bimodule over $(\mathfrak{g}, [~,~]_1)$;
\item[ii.] $(M, l_2, r_2)$ is a bimodule over $(\mathfrak{g}, [~,~]_2)$;
\item[iii.] For all $x,y \in \mathfrak{g}$, and $m\in M$, the following compatibility conditions hold:
\begin{itemize}
\item[(a)] $r_1(x, l_2(y,m) +r_2 (x, l_1(y, m))= l_1([x,y]_2, m) + l_2([x,y]_1, m)- r_1(l_2(x,m),y)-r_2(l_1(x,m),y)$;
\item[(b)] $l_1(x, r_2(m,y))+ l_2(x, r_1(m,y)) = r_1(l_2(x,m),y)+r_1(l_2(x,m),y)-l_1([x,y]_2,m)-l_2([x,y]_1,m)$;
\item[(c)] $r_1(m, [x,y]_2) + r_2(m, [x,y]_1)= r_1(r_2(m,x), y) + r_2(r_1(m,x), y)- r_1(r_2(m,y),x)- r_2(r_1(m,y),x).$
\end{itemize}
\end{itemize}
\end{defn}

\begin{exam}
Any compatible Leibniz algebra $(\mathfrak{g}, [~,~]_1, [~,~]_2)$ is a compatible bimodule over itself by considering $l_1=r_1=[~,~]_1$ and $l_2=r_2=[~,~]_2$.
\end{exam}

\begin{exam}
If $(\mathfrak{g}, [~,~]_1, [~,~]_2)$ is a compatible Leibniz algebra. Then we know that $(\mathfrak{g}, \lambda_1[~,~]_1 + \lambda_2[~,~]_2)$ is also a Leibniz algebra. Let  $(M, l_1, r_1, l_2, r_2)$ be a compatible $\mathfrak{g}$-bimodule. Then it is a routine work to check that $(M, l_1+l_2, r_1+r_2)$ is a bimodule over the Leibniz algebra $(\mathfrak{g}, \lambda_1[~,~]_1 + \lambda_2[~,~]_2)$.
\end{exam}

\begin{prop}\label{semi}
Let $(\mathfrak{g}, [~,~]_1, [~,~]_2)$ be a compatible Leibniz algebra and $(M, l_1, r_1, l_2, r_2)$ be a compatible $\mathfrak{g}$-bimodule. Then the direct sum $\mathfrak{g}\oplus M$ has a compatible Leibniz algebra structure with the following binary operations:
\begin{align*}
&[(x,m), (y,n)]^{1} = ([x,y]_1, l_1(x,n) +r_1(m,y));\\
& [(x,m), (y,n)]^{2} = ([x,y]_2, l_2(x,n) +r_2(m,y)),
\end{align*}
for all $(x,m), (y,n) \in \mathfrak{g} \oplus M$. This structure is called the semi-direct product.
\end{prop}

\subsection{Low-dimensional compatible Leibniz algebras}
In this subsection, we explore the classification of complex  Leibniz algebras in dimension 2 and 3 to provide all compatible pairs.  We refer for the classifications to \cite{Rakhimov}. 
\subsubsection{2-Dimensional compatible Leibniz algebras}
There are three unabelain non-isomorphic 2-dimensional Leibniz algebras.   They are given with respect to basis $\{ e_1,e_2\}$ by 
\begin{enumerate}
\item[$L_1 :$] $ [e_1,e_2]=e_2,\; [e_2,e_1]=-e_2$ (solvable Lie algebra);
\item[$L_2 :$] $ [e_1,e_1]=e_2$ (nilpotent  Leibniz algebra);
\item[$ L_3 :$] $ [e_1,e_1]=e_2,\; [e_2,e_1]=e_2$ (solvable Leibniz algebra).
\end{enumerate}
\begin{prop}
There is, up to isomorphism,  only one   pair of 2-dimensional compatible Leibniz algebras. It is given by $(L_2,L_3 )$.

\end{prop}
 
\subsubsection{3-Dimensional compatible Leibniz algebras}
Every non-abelian  3-dimensional Leibniz algebra is isomorphic to one of the following algebras with respect to basis a $\{ e_1,e_2, e_3\}$:
\begin{enumerate}
\item[$L_1(\alpha): $] $  [e_1,e_3] = \alpha e_1,  [e_2,e_3] = e_1 + e_2,
[e_3,e_3] = e_1,$ $\forall \alpha \in \mathbb{C}^*$ (solvable Leibniz algebra);
\item[$L_2: $] $  [e_3,e_3] =  e_1, [e_2,e_3] = e_1 + e_2,$  (solvable Leibniz algebra);
\item[$L_3: $] $  [e_1,e_2] = e_3, [e_1,e_3] = -2e_3, 
[e_2,e_1] = -e_3, [e_2,e_3] = 2e_3,
[e_3,e_1] = 2e_3, [ e_3,e_2] = -2e_3,$  (simple Lie  algebra);
\item[$L_4(\alpha): $] $  [e_1,e_3] = \alpha e_1,  [e_2,e_3] = - e_2, [e_3,e_2] = e_2, 
[e_3,e_3] = e_1,$ $\forall \alpha \in \mathbb{C}$ (solvable Leibniz algebra);
\item[$L_5: $] $  [e_1,e_3] =  e_1, [e_2,e_3] = e_1 ,[e_3,e_3] = e_1,$  (solvable Leibniz algebra);
\item[$L_6: $] $  [e_1,e_3] =  e_2, [e_3,e_3] = e_1,$  (nilpotent Leibniz algebra);
\item[$L_7: $] $  [e_1,e_2] =  e_1, [e_1,e_3] = e_1, [e_3,e_2] =  e_1, [e_3,e_3] = e_1,$  (solvable Leibniz algebra);
\item[$L_8: $] $  [e_1,e_1] =  e_2, [e_2,e_1] = e_2,$  (solvable Leibniz algebra);
\item[$L_9(\alpha): $] $  [e_1,e_2] =e_2,  [e_1,e_3] = \alpha e_3, [e_2,e_1] =- e_2, 
[e_3,e_1] =-\alpha  e_3,$ $\forall \alpha \in \mathbb{C}-\{0,1\}$ (solvable Lie algebra);
\item[$L_{10}: $] $  [e_1,e_2] =  e_2, [e_2,e_1] =- e_2,$  (solvable Lie algebra);
\item[$L_{11}: $] $  [e_1,e_2] =  e_2, [e_1,e_3] = e_2+e_3, [e_2,e_1] = - e_2, [e_3,e_1] =- e_2-e_3,$  (solvable Lie algebra);
\item[$L_{12}(\alpha): $] $  [e_2,e_2] =  e_1,  [e_2,e_3] =  e_1, 
[e_3,e_3] = \alpha e_1,$ $\forall \alpha \in \mathbb{C}$ (nilpotent Leibniz algebra);
\item[$L_{13}: $] $  [e_2,e_2] =  e_1, [e_2,e_3] = e_1 ,[e_3,e_2] = e_1,$  (associative commutative nilpotent  Leibniz algebra);
\item[$L_{14}: $] $  [e_1,e_3] =  e_1, [e_2,e_3] = e_2 ,[e_3,e_3] = e_1,$  (solvable Leibniz algebra);
\item[$L_{15}: $] $  [e_1,e_1] =  e_2, $  (associative commutative nilpotent  Leibniz algebra);
\item[$L_{16}: $] $  [e_1,e_2] =  e_2, [e_1,e_3] = e_3, [e_2,e_1] = - e_2, [e_3,e_1] =-e_3,$  (solvable Lie algebra);
\item[$L_{17}: $] $  [e_1,e_2] =  e_3,  [e_2,e_1] = - e_3, $  (nilpotent Lie algebra).
\end{enumerate}
\begin{prop}The 3-dimensional compatible Leibniz algebras are given by the pairs:
\begin{eqnarray*}
& (L_1,L_2 ),\; (L_1,L_5 ),\; (L_1,L_6 ),\; (L_1,L_{14} ),\; (L_2,L_5 ),\;  (L_2,L_6 ),\; (L_2,L_{14} ),\;\\
& (L_4,L_{13} ) \text{ for } \alpha=-2,\; (L_5,L_{6} ),\;(L_5,L_{14} ),\;(L_6,L_{14} ),\; (L_7,L_{12} ) \text{ for } \alpha=0,\;\\
& (L_8,L_{15} ),\;(L_9,L_{10} ),\;(L_9,L_{11} ),\;(L_9,L_{16} ),\;(L_9,L_{17} ),\;
(L_{10},L_{11} ),\;(L_{10},L_{16} ),\;\\ &(L_{11},L_{16} ),\;(L_{11},L_{17} ),\;(L_{12},L_{13} ),\;(L_{16},L_{17} ).
\end{eqnarray*}
\end{prop}

\section{Maurer-Cartan characterization of compatible Leibniz algebras}\label{sec5}
In this Section, we characterize compatible Leibniz algebras as  Maurer-Cartan elements of a suitable graded Lie algebra.

Let $(\mathfrak{g} = \oplus_n \mathfrak{g}^n, [~,~])$ be a graded Lie algebra. A Maurer-Cartan element of $\mathfrak{g}$ is an element $\alpha \in \mathfrak{g}^1$ such that
\begin{align*}
[\alpha, \alpha] = 0.
\end{align*}

It is well-known that if $\alpha$ is a Maurer-Cartan element then we get a degree $1$ coboundary map $d_\alpha : =[ \alpha , - ]$ on $\mathfrak{g}$. Therefore, we get a differential graded Lie algebra $(\mathfrak{g}, [~,~],d_\alpha)$. For any $\alpha' \in \mathfrak{g}^1$, the sum $\alpha+ \alpha'$ is a Maurer-Cartan element of $\mathfrak{g}$ if and only if $\alpha'$ satisfies
\begin{align*}
&[\alpha+\alpha', \alpha+\alpha']=0\\
&[\alpha,\alpha] + [\alpha, \alpha'] + [\alpha', \alpha] + [\alpha', \alpha']=0\\
&2[\alpha, \alpha'] + [\alpha', \alpha']=0\\
&d_\alpha (\alpha') + \frac{1}{2} [\alpha', \alpha' ] = 0.
\end{align*}

\begin{defn}
Two Maurer-Cartan elements $\alpha_1$ and $\alpha_2$ are said to be compatible if they satisfy  $[\alpha_1, \alpha_2] = 0$. In this case, we say that $(\alpha_1, \alpha_2)$ is a compatible pair of Maurer-Cartan elements of $\mathfrak{g}$.
\end{defn}

We define $\mathfrak{g}_{com} = \oplus_{n \geq 0} (\mathfrak{g}_{com})^n$, where
\begin{align*}
(\mathfrak{g}_{com})^0 = \mathfrak{g}^0 ~~~~ \text{ and } ~~~~ (\mathfrak{g}_{com})^n = \underbrace{\mathfrak{g}^n \oplus \cdots \oplus \mathfrak{g}^n}_{(n+1) \text{ times}}, ~\text{ for } n \geq 1.
\end{align*}
Let $[ ~,~ ]_c : (\mathfrak{g}_{com})^m \times (\mathfrak{g}_{com})^n \rightarrow (\mathfrak{g}_{com})^{m+n}$, for $m,n \geq 0$, be the degree $0$ bracket defined by
\begin{align}\label{br}
 &[(h_1, \ldots, h_{m+1}), (k_1, \ldots, k_{n+1}) ]_c :=  \\
&= \big( [h_1, k_1],~ [h_1, k_2] + [h_2, k_1],~ \ldots, \underbrace{[h_1, k_i] + [h_2, k_{i-1}] + \cdots + [h_i, k_1]}_{i\text{-th place}},~ \ldots, [h_{m+1}, k_{n+1}] \big), \nonumber
\end{align}
for $(h_1, \ldots, h_{m+1}) \in (\mathfrak{g}_{com})^m$ and $(k_1, \ldots, k_{n+1}) \in (\mathfrak{g}_{com})^n.$

\begin{prop}
\begin{itemize}
\item[(i)] $(\mathfrak{g}_{com}, [ ~, ~ ]_c)$ is a graded Lie algebra. Moreover, the map $\psi : \mathfrak{g}_{com} \rightarrow \mathfrak{g}$ defined by
\begin{align*}
\psi (h ) =~& h, ~ \text{ for } h \in (\mathfrak{g}_{com})^0 = \mathfrak{g}^0,\\
\psi ((h_1, \ldots, h_{n+1})) =~& h_1 + \cdots + h_{n+1},~ \text{ for } (h_1, \ldots, h_{n+1}) \in (\mathfrak{g}_{com})^n
\end{align*}
is a morphism of graded Lie algebras.
\item[(ii)] A pair $(\alpha_1, \alpha_2)$ of elements of $\mathfrak{g}^1$ is a compatible pair of Maurer-Cartan elements of $\mathfrak{g}$ if and only if $(\alpha_1, \alpha_2) \in (\mathfrak{g}_{com})^1 = \mathfrak{g}^1 \oplus \mathfrak{g}^1$ is a Maurer-Cartan element in the graded Lie algebra $(\mathfrak{g}_{com}, [ ~, ~ ]_c).$
\end{itemize}
\end{prop}

\begin{proof}
(i) For $(h_1, \ldots, h_{m+1}) \in (\mathfrak{g}_{com})^m$, $(k_1, \ldots, k_{n+1}) \in (\mathfrak{g}_{com})^n$ and $(l_1, \ldots, l_{p+1}) \in (\mathfrak{g}_{com})^p$,
\begin{align*}
&[ (h_1, \ldots, h_{m+1}), [ (k_1, \ldots, k_{n+1}) ,  (l_1, \ldots, l_{p+1})]_c ]_c \\
&= [ (h_1, \ldots, h_{m+1}) , \big( [k_1, l_1], \ldots, \underbrace{\sum_{q+r = i+1}  [k_q, l_r]}_{i\text{-th place}}, \ldots, [k_{n+1}, l_{p+1}] \big) ]_c \\
&= \big( [h_1, [k_1, l_1]], \ldots, \underbrace{ \sum_{p+q+r= i+2} [h_p, [k_q, l_r]]}_{i\text{-th place}}, \ldots, [h_{m+1}, [k_{n+1}, l_{p+1}]]      \big)\\
&= \bigg( [[h_1, k_1], l_1] + (-1)^{mn} ~ [k_1, [h_1, l_1]]~, \ldots, \underbrace{\sum_{p+q+r = i+2} [[h_p, k_q], l_r] + (-1)^{mn} ~[k_q, [h_p, l_r]]}_{i\text{-th place}}, \\
& \qquad \qquad \qquad \ldots,  [[h_{m+1}, k_{n+1}], l_{p+1}] + (-1)^{mn} ~ [k_{n+1}, [h_{m+1}, l_{p+1}]]    \bigg) \\
&= [ [ (h_1, \ldots, h_{m+1}), (k_1, \ldots, k_{n+1})]_c, (l_1, \ldots, l_{p+1})]_c  \\
& \qquad \qquad \qquad + (-1)^{mn}~ [  (k_1, \ldots, k_{n+1}), [ (h_1, \ldots, h_{m+1}), (l_1, \ldots, l_{n+1}) ]_c ]_c.
\end{align*}
 We also have
\begin{align*}
\psi [ (h_1, \ldots, h_{m+1}), (k_1, \ldots, k_{n+1})]_c =~& \sum_{i=1}^{m+n+1} \sum_{q+r = i+1} [h_q, k_r] \\
=~&[h_1 + \cdots + h_{m+1}, ~k_1 + \cdots + k_{n+1}] \\
=~& [\psi (h_1, \ldots, h_{m+1}), \psi (k_1, \ldots, k_{n+1}) ],
\end{align*}
which completes the second part.

(ii) For a pair $(\alpha_1, \alpha_2)$ of elements of $\mathfrak{g}^1$, we have
\begin{align*}
[ (\alpha_1, \alpha_2), (\alpha_1, \alpha_2)]_c = ( [\alpha_1, \alpha_1], [\alpha_1, \alpha_2] + [\alpha_2, \alpha_1], [\alpha_2, \alpha_2]) = ([\alpha_1, \alpha_1], 2[\alpha_1, \alpha_2], [\alpha_2, \alpha_2]).
\end{align*}
Therefore, $(\alpha_1, \alpha_2) \in (\mathfrak{g}_{com})^1$ is a Maurer-Cartan element in $\mathfrak{g}_{com}$ if and only if $(\alpha_1, \alpha_2)$ is a pair of compatible Maurer-Cartan elements in $\mathfrak{g}$.
\end{proof}

Thus, from the graded Lie bracket (defined in Section \ref{sec2}) and the above proposition, we get the following.

\begin{thm}
Let $L$ be a vector space.
\begin{itemize}
\item[(i)] Then $C^{\ast + 1}_{com} (L, L) := \oplus_{n \geq 0} C^{n+1}_{com} (L,L)$, where
\begin{align*}
&C^1_{com} (L,L) = C^1(L,L);\\
&C^{n+1}_{com} (L,L) = \underbrace{C^{n+1} (L,L) \oplus \cdots \oplus C^{n+1}(L,L)}_{(n+1) \text{ times}}, ~\text{ for } n \geq 1
\end{align*}
is a graded Lie algebra with bracket given by (\ref{br}) where $[~,~]$ is replaced by $[~,~]_c$. Moreover, the map
\begin{align}\label{the-phi}
\psi : C^{\ast + 1}_{com} (L,l) \rightarrow C^{\ast + 1 } (L,L),~ (h_1, \ldots, h_{n+1}) \mapsto h_1 + \cdots + h_{n+1}, \text{ for } n \geq 0
\end{align}
is a morphism of graded Lie algebras.
\item[(ii)] A pair $(m_1, m_2) \in C^2_{com}(L,L)= C^2(L,L) \oplus C^2(L,L)$ defines a compatible Leibniz algebra structure on $L$ if and only if $(m_1, m_2) \in C^2_{com}(L,L)$ is a Maurer-Cartan element in the graded Lie algebra $(C^{\ast + 1}_{com} (L,L), [ ~, ~ ]_c)$.
\end{itemize}
\end{thm}

Let $(\mathfrak{g},m_1, m_2)$ be a compatible Leibniz algebra. Then there is a degree $1$ coboundary map
\begin{align}\label{d-mu}
d_{(m_1, m_2)} := [ (m_1, m_2), ~ ] : C^n_{com} (\mathfrak{g}, \mathfrak{g}) \rightarrow C^{n+1}_{com} (\mathfrak{g}, \mathfrak{g}), \text{ for } n \geq 1
\end{align}
which makes $(C^{\ast + 1}_{com} (\mathfrak{g}, \mathfrak{g}), [~, ~]_c, d_{(m_1, m_2)} )$ into a differential graded Lie algebra.

\section{Cohomology of compatible Leibniz algebras}\label{sec6}
In this section, we introduce the cohomology of a compatible Leibniz algebra with self representation.

Let $(\mathfrak{g}, m_1, m_2)$ be a compatible Leibniz algebra and $M = (M, l_1, r_1, l_2, r_2)$ be a representation of $\mathfrak{g}$. Let
\begin{align*}
\delta^n_1 : C^n(\mathfrak{g}, M) \rightarrow C^{n+1} (\mathfrak{g}, M), ~ n \geq 0,
\end{align*}
denotes the coboundary operator for the Leibniz cohomology of $(\mathfrak{g}, m_1)$ with coefficients in $(M, l_1, r_1)$, and
\begin{align*}
\delta^n_2 : C^n(\mathfrak{g}, M) \rightarrow C^{n+1} (\mathfrak{g}, M), ~ n \geq 0,
\end{align*}
denotes the coboundary operator for the Leibniz cohomology of $(\mathfrak{g}, m_2)$ with coefficients $(M, l_2, r_2)$. Then, we have
\begin{align*}
(\delta_1)^2 = 0 \qquad \text{ and } \qquad (\delta_2)^2 = 0.
\end{align*}
Now we give the interpretation of $\delta_1$ and $\delta_2$ in terms of two Leibniz algebra structures on $\mathfrak{g} \oplus M$ given in Proposition \ref{semi}. Let $\mu_1, \mu_2 \in C^2 (\mathfrak{g} \oplus M, \mathfrak{g} \oplus M)$ denote the elements corresponding to the Leibniz products on $\mathfrak{g}\oplus M$.

Note that any map $f \in C^n(\mathfrak{g}, M)$ can be lifted to a map $\widetilde{f} \in C^n (\mathfrak{g} \oplus M, \mathfrak{g} \oplus M)$ by
\begin{align*}
\widetilde{h} \big( (x_1, m_1), \ldots, (x_n, m_n)  \big) = \big( 0, h (x_1, \ldots, x_n ) \big),
\end{align*}
for $(x_i, m_i) \in \mathfrak{g} \oplus M$ and $i=1, \ldots, n$. Moreover, we have  $\widetilde{h} = 0$ if and only if $h=0$. With all these notations, we have
\begin{align*}
\widetilde{(\delta_1 h )} = (-1)^{n-1}~[ \mu_1, \widetilde{h}]  \qquad \text{ and } \qquad \widetilde{(\delta_2 h )} = (-1)^{n-1}~[ \mu_2, \widetilde{h}],
\end{align*}
for $h \in C^n( \mathfrak{g}, M)$. 

\begin{prop}\label{delta-comp}
The coboundary operators $\delta_1$ and $\delta_2$ satisfy
\begin{align*}
\delta_1 \circ \delta_2 + \delta_2 \circ \delta_1 = 0.
\end{align*}
\end{prop}

\begin{proof}
For any $h \in C^n ( \mathfrak{g},M)$, we have
\begin{align*}
&\widetilde{( \delta_1 \circ \delta_2 + \delta_2 \circ \delta_1)(h)}  \\
&= (-1)^{n} ~  [\mu_1, \widetilde{\delta_2 h}] ~+~ (-1)^{n} ~ [\mu_2, \widetilde{\delta_1 h}] \\
&= - [\mu_1, [\mu_2, \widetilde{h}] ]  - [\mu_2, [\mu_1, \widetilde{h}]  \\
&= - [[\mu_1, \mu_2], \widetilde{h}]  = 0 ~~~~ \qquad (\text{because} ~[\mu_1, \mu_2] = 0).
\end{align*}
Therefore, it follows that  $  ( \delta_1 \circ \delta_2 + \delta_1 \circ \delta_1)(h) = 0$. Hence, $\delta_1 \circ \delta_2 + \delta_2 \circ \delta_1 = 0$.
\end{proof}

The compatibility condition of the above proposition leads to cohomology associated with a compatible Leibniz algebra with coefficients in a compatible representation. Let $\mathfrak{g}$ be a compatible Leibniz algebra and $M$ be a representation of it. We define the $n$-th cochain group $C^n_{com} (\mathfrak{g}, M)$, for $n \geq 0$, by
\begin{align*}
C^0_{com} (\mathfrak{g}, M) :=~& \{ m \in M ~|~ x \cdot_1 m - m \cdot_1 x = x \cdot_2 m - m \cdot_2 x, ~\forall x \in \mathfrak{g} \},\\
C^n_{com} (\mathfrak{g}, M) :=~& \underbrace{C^n (\mathfrak{g}, M) \oplus \cdots \oplus C^n (\mathfrak{g}, M)}_{n \text{ copies}}, ~ \text{ for } n \geq 1.
\end{align*}
Define a map $\delta_c : C^n_{com} (\mathfrak{g}, M) \rightarrow C^{n+1}_{com} (\mathfrak{g}, M)$, for $n \geq 0$, by
\begin{align}
\delta_c (m) (a) :=~&  x \cdot_1 m - m \cdot_1 x = x \cdot_2 m - m \cdot_2 x, \text{ for } m \in C^0_{com} (\mathfrak{g}, M) \text{ and } x \in \mathfrak{g}, \label{dc-1}\\
\delta_c (h_1, \ldots, h_n ) :=~& ( \delta_1 h_1, \ldots, \underbrace{\delta_1 h_i + \delta_2 h_{i-1}}_{i-\text{th place}}, \ldots, \delta_2 h_n),\label{dc-2}
\end{align}
for $(h_1, \ldots, h_n ) \in C^n_{com} (\mathfrak{g}, M)$. 



\begin{prop}
The map $\delta_c$ is a coboundary operator, i.e., $(\delta_c)^2 = 0$.
\end{prop}

\begin{proof}
For $m \in C^0_c (\mathfrak{g},M)$, we have
\begin{align*}
(\delta_c)^2 (m) = \delta_c ( \delta_c m ) =~& (\delta_1 \delta_c m~,\delta_2 \delta_c m ) \\
=~&  ( \delta_1 \delta_1 m~, \delta_2 \delta_2 m) = 0.
\end{align*}
Moreover, for any $(h_1, \ldots, h_n) \in C^n_c (\mathfrak{g},M)$, $n \geq 1$, we have
\begin{align*}
(\delta_c)^2 (h_1, \ldots, h_n)
&= \delta_c \big(    \delta_1 h_1, \ldots,  \delta_1 h_i +  \delta_2 h_{i-1}, \ldots,  \delta_2 h_n \big) \\
&= \big(    \delta_1  \delta_1 h_1~,  \delta_2  \delta_1 h_1 +  \delta_1  \delta_2 h_1 +  \delta_1  \delta_1 h_2~, \ldots, \\
& \qquad \underbrace{   \delta_2  \delta_2 h_{i-2} +  \delta_2 \delta_1 h_{i-1} + \delta_1 \delta_2 h_{i-1} + \delta_1  \delta_1 h_i  }_{3 \leq i \leq n-1}~, \ldots, \\
& \qquad  \delta_2 \delta_2  h_{n-1} + \delta_2 \delta_1 h_n +  \delta_1 \delta_2 h_n ~,~  \delta_2  \delta_2 h_n \big) \\
&= 0 ~~~\quad (\text{from Proposition } \ref{delta-comp}).
\end{align*}
This proves that $(\delta_c)^2 = 0$.
\end{proof}

Thus, we have a cochain complex $\{ C^\ast_{com} (\mathfrak{g}, M), \delta_c \}$. Let $Z^n_{com} (\mathfrak{g}, M)$ denote the space of $n$-cocycles and $B^n_{com} (\mathfrak{g}, M)$ the space of $n$-coboundaries. Then we have $B^n_{com} (\mathfrak{g}, M) \subset Z^n_{com} (\mathfrak{g}, M)$, for $n \geq 0$. The corresponding quotient groups
\begin{align*}
H^n_{com} (\mathfrak{g}, M) := \frac{ Z^n_{com} (\mathfrak{g}, M) }{ B^n_{com} (\mathfrak{g}, M)}, \text{ for } n \geq 0
\end{align*}
are called the cohomology of the compatible Leibniz algebra $\mathfrak{g}$ with coefficients in the representation $M$.

\section{Formal deformation theory of compatible Leibniz algebras}\label{sec7}
In this section, we study formal deformation theory of compatible Leibniz algebras. In this study, we will closely follow the deformation theory by Gerstenhaber \cite{gers, gers2} for associative algebras. It is based on formal power series  in variable $t$, $\mathbb{K}[[t]]$. Any vector space $\mathfrak{g}$ extends naturally  to a formal space $\mathfrak{g}[[t]]=\{ \sum_{i\leq 0} a_i t_i,\;  a_i\in \mathfrak{g}\}$.
\begin{defn}\label{deform defn}
Let $(\mathfrak{g}, m_1, m_2)$ be a compatible Leibniz algebra. A one-parameter formal deformation of $(\mathfrak{g}, m_1, m_2)$ is a triple $(\mathfrak{g}[[t]], m_{1,t}, m_{2,t})$, where
\begin{align*}
m_{1,t}, m_{2,t} : \mathfrak{g}[[t]]\times \mathfrak{g}[[t]] \to \mathfrak{g}[[t]]
\end{align*}
are $\mathbb{K}[[t]]$-bilinear maps of the form
$$m_{1,t} = \sum_{i\geq 0} m_{1,i} t^i,~ m_{2,t} = \sum_{i\geq 0} m_{2,i} t^i,$$
such that
\begin{itemize}
\item[(i)]  $m_{1,i}, m_{2,i} : \mathfrak{g}\times \mathfrak{g} \to \mathfrak{g}$ are $\mathbb{K}$-bilinear maps for all $i\geq 0$.

\item[(ii)] $m_{1,0} = m_1, m_{2,0} = m_2$ are the original bracket operations respectively.

\item[(iii)] $(\mathfrak{g}[[t]], m_{1,t})$ and $(\mathfrak{g}[[t]], m_{2,t})$ are both Leibniz algebras, that is, for all $x,y,z\in \mathfrak{g}$, we have
\begin{align}
&m_{1,t}(x, (m_{1,t}(y, z)) = m_{1,t}(m_{1,t}(x, y),z)- m_{1,t}(m_{1,t}(x, z),y),\\
&m_{2,t}(x, (m_{2,t}(y, z)) = m_{2,t}(m_{2,t}(x, y),z)- m_{2,t}(m_{2,t}(x, z),y).\label{deform ass-2}
\end{align}

\item[(iv)] $(\mathfrak{g}[[t]], m_{1,t}, m_{2,t})$ satisfies the following compatibility conditions:
\begin{align}\label{deform com}
 &m_{2,t}(x, m_{1 ,t}(y,z)) + m_{1,t}(x, m_{2 ,t}(y,z)) \\
 &= m_{2,t}(m_{1,t}(x, y),z) + m_{1,t}(m_{2,t}(x, y),z)- m_{2,t}(m_{1,t}(x, z),y)- m_{1,t}(m_{2,t}(x, z),y)\nonumber 
  \end{align}
 for all $x,y,z \in \mathfrak{g}$.
\end{itemize}
\end{defn}

 Equations (4.1) and (\ref{deform ass-2}) are equivalent to the following equations. For all $n\geq 0$, we have
\begin{align}
&\sum_{i+j=n} \big(m_{1,i}(x, (m_{1,j}(y, z))- m_{1,i}(m_{1,j}(x, y),z)+ m_{1,i}(m_{1,j}(x, z),y)\big)=0 \label{deform ass-1-1}\\
&\sum_{i+j=n} \big(m_{2,i}(x, (m_{2,j}(y, z)) - m_{2,i}(m_{2,j}(x, y),z) + m_{2,i}(m_{2,j}(x, z),y)\big)=0.\label{deform ass-2-1}
\end{align}

Equivalently, we can write Equations \ref{deform ass-1-1} and \ref{deform ass-2-1} as follows:

\begin{align}
&\sum_{i+j=n} [m_{1,i}, m_{1,j}] = 0,\label{deform ass-1-3}\\
&\sum_{i+j=n} [m_{2,i}, m_{2,j}] = 0.\label{deform ass-2-3}
\end{align}

The condition \ref{deform com} is equivalent to the following equations. For all $n\geq 0$, we have
\begin{align}
&\sum_{i+j=n}\big( m_{2,i}(x, m_{1 ,j}(y,z)) + m_{1,i}(x, m_{2 ,j}(y,z))\big)\label{deform com 1} \\
\nonumber &= \sum_{i+j=n}\big( m_{2,i}(m_{1,j}(x, y),z) + m_{1,i}(m_{2,j}(x, y),z)- m_{2,i}(m_{1,j}(x, z),y)- m_{1,i}(m_{2,j}(x, z),y)\big).
\end{align}

For all $n\geq 0$, we can re-write the Equation \ref{deform com 1} as follows:

\begin{align}
&\sum_{i+j=n} [m_{1,i}, m_{2,j}] = 0. \label{deform com 2}
\end{align}

Therefore, using Equations \ref{deform ass-1-3}, \ref{deform ass-2-3}, and \ref{deform com 2}, we can say that $(\mathfrak{g}[[t]], m_{1,t}, m_{2,t})$ is a one-parameter formal deformation of the compatible Leibniz algebra $\mathfrak{g}$ if for all $n\geq 0$, and $x,y,z\in \mathfrak{g}$, it satisfies the following equations:
\begin{align*}
&\sum_{i+j=n} [m_{1,i}, m_{1,j}] = 0,\\
&\sum_{i+j=n} [m_{2,i}, m_{2,j}] = 0,\\
&\sum_{i+j=n} [m_{1,i}, m_{2,j}] = 0.
\end{align*}

For $n=0$, we have 
$$[m_{1,0}, m_{1,0}] = 0, ~ [m_{2,0}, m_{2,0}]= 0,~[m_{1,0}, m_{2,0}] = 0.$$
This is same as
$$[m_{1}, m_{1}] = 0, ~ [m_{2}, m_{2}] = 0,~[m_{1}, m_{2}] = 0.$$
Note that the above relations are nothing but original Leibniz identities and compatibility relation.

For $n=1$, we have
\begin{align*}
& [m_{1,1}, m_{1,0}] + [m_{1,0}, m_{1,1}]+ [m_{1,0}, m_{1,0}] = 0,\\
&[m_{2,1}, m_{2,0}] + [m_{2,0}, m_{2,1}] + [m_{2,0}, m_{2,0}] = 0,\\
&[m_{1,1}, m_{2,0}] + [m_{1,0}, m_{2,1}] + [m_{1,0}, m_{2,0}] = 0.
\end{align*}

Equivalently, we have 

\begin{align*}
& [m_{1,1}, m_{1}] + [m_{1}, m_{1,1}]+ [m_{1}, m_{1}] = 0,\\
&[m_{2,1}, m_{2}] + [m_{2}, m_{2,1}] + [m_{2}, m_{2}] = 0,\\
&[m_{1,1},m_{2}] + [m_{1}, m_{2,1}] + [m_{1}, m_{2}] = 0.
\end{align*}
As $[m_1,m_1]=0$ and $[m_2,m_2]=0$, we have
\begin{align*}
& [m_{1,1}, m_{1}] = 0,\\
&[m_{2,1},m_{2}] = 0,\\
&[m_{1,1}, m_{2}] + [m_{1}, m_{2,1}]  = 0.
\end{align*}
Therefore, 
$$\delta^{2}_c(m_{1,1}, m_{2,1})=0.$$
Thus, $(m_{1,1}, m_{2,1})$ is a $2$-cocyle in the cohomology of the compatible Leibniz algebra $\mathfrak{g}$ with coefficients in itself. This pair $(m_{1,1}, m_{2,1})$ is called the infinitesimal of the deformation. This means the infinitesimal of the deformation is a $2$-cocycle. More generally, we have the following definition.

\begin{defn}
If $(m_{1,n}, m_{2,n})$ is the first non-zero term after $(m_{1,0}, m_{2,0})=(m_1, m_2)$ of the formal deformation $(m_{1,t}, m_{2,t})$, then we say that $(m_{1,n}, m_{2,n})$ is the $n$-infinitesimal of the deformation.
\end{defn}

\begin{thm}\label{infintesimal}
The $n$-infinitesimal is a $2$-cocycle.
\end{thm}
The proof is similar of showing that the infinitesimal is a $2$-cocycle.

\subsection{Equivalent deformation and cohomology}
Let $\mathfrak{g}_t=(\mathfrak{g},m_{1,t},m_{2,t})$ and $\mathfrak{g}'_t=(\mathfrak{g},m'_{1,t},m'_{2,t})$ be two one-parameter compatible Leibniz algebra deformations of $(\mathfrak{g},m_1, m_2)$, where $m_{1,t} = \sum_{i\geq 0} m_{1,i} t^i,~ m_{2,t} = \sum_{i\geq 0} m_{2,i} t^i,$ and $m'_{1,t} = \sum_{i\geq 0} m'_{1,i} t^i,~ m'_{2,t} = \sum_{i\geq 0} m'_{2,i} t^i$. 
\begin{defn}
 Two deformations $\mathfrak{g}_t$ and $\mathfrak{g}\rq_t$ are said to be equivalent if there exists a $\mathbb{K}[[t]]$-linear isomorphism $\Phi_t:\mathfrak{g}[[t]]\to \mathfrak{g}[[t]]$ of the form $\Phi_t=\sum_{i\geq 0}\phi_it^i$, where $\phi_0=id$ and $\phi_i:\mathfrak{g}\to \mathfrak{g}$ are $\mathbb{K}$-linear maps such that the following relations holds:
 \begin{align}\label{equivalent-1}
 &\Phi_t\circ m_{1,t}\rq=m_{1,t}\circ (\Phi_t\otimes \Phi_t),\\
 \label{equivalent-2}&\Phi_t\circ m_{2,t}\rq=m_{2,t}\circ (\Phi_t\otimes \Phi_t).
 \end{align}
 \end{defn}
 \begin{defn}
 A deformation $(m_{1,t}, m_{2,t})$ of a compatible Leibniz algebra $\mathfrak{g}$ is called trivial if $(m_{1,t}, m_{2,t})$ is equivalent to the deformation $(m_{1,0}, m_{2,0})$, which is the same as the undeformed one. A compatible Leibniz algebra $\mathfrak{g}$ is called rigid if it has only trivial deformation up to equivalence.
 \end{defn}
 Equations (\ref{equivalent-1}-\ref{equivalent-2}) may be written as
 \label{equivalent 11}
\begin{align}
& \Phi_t(m\rq_{1,t}(x,y))=m_{1,t}(\Phi_t(x),\Phi_t(y)),\\
& \Phi_t(m\rq_{2,t}(x,y))=m_{2,t}(\Phi_t(x),\Phi_t(y)),\,\,\,\text{for all}~ x,y\in \mathfrak{g}.
 \end{align} 
 Note that the above equations are equivalent to the following equations:
 \begin{align}
 &\sum_{i\geq 0}\phi_i\bigg( \sum_{j\geq 0}m\rq_{1,j}(x,y)t^j \bigg)t^i=\sum_{i\geq 0}m_{1,i}\bigg( \sum_{j\geq 0}\phi_j(x)t^j,\sum_{k\geq 0}\phi_k(y)t^k \bigg)t^i,\\
  &\sum_{i\geq 0}\phi_i\bigg( \sum_{j\geq 0}m\rq_{2,j}(x,y)t^j \bigg)t^i=\sum_{i\geq 0}m_{2,i}\bigg( \sum_{j\geq 0}\phi_j(x)t^j,\sum_{k\geq 0}\phi_k(y)t^k \bigg)t^i.\\
 \end{align}
 This is same as the following equations:
 \begin{align}
 \label{equivalent 10}&\sum_{i,j\geq 0}\phi_i(m\rq_{1,j}(x,y))t^{i+j}=\sum_{i,j,k\geq 0}m_{1,i}(\phi_j(x),\phi_k(y))t^{i+j+k},\\
 &\sum_{i,j\geq 0}\phi_i(m\rq_{2,j}(x,y))t^{i+j}=\sum_{i,j,k\geq 0}m_{2,i}(\phi_j(x),\phi_k(y))t^{i+j+k}.
 \end{align}
Using $\phi_0=Id$ and comparing constant terms on both sides of the above equations, we have
 \begin{align*}
 &m\rq_{1,0}(x,y)=m_1(x,y),\\
 &m\rq_{2,0}(x,yy)=m_2(x,y).
  \end{align*}
  Now comparing coefficients of $t$, we have
  \begin{align}\label{equivalent main}
&m\rq_{1,1}(x,y)+\phi_1(m\rq_{1,0}(x,y))=m_{1,1}(x,y)+m_{1,0}(\phi_1(x),y)+m_{1,0}(x,\phi_1(y)),\\
&m\rq_{2,1}(x,y)+\phi_1(m\rq_{2,0}(x,y))=m_{2,1}(x,y)+m_{2,0}(\phi_1(x),y)+m_{2,0}(x,\phi_1(y)).\label{equivalent main 1}
  \end{align}
  The Equations (\ref{equivalent main})-(\ref{equivalent main 1}) are same as
  \begin{align*}
& m\rq_{1,1}(x,y)-m_{1,1}(x,y)=m_1(\phi_1(x),y)+m_1(x,\phi_1(y))-\phi_1(m_1(x,y))=\delta_1 \phi_1(x,y).\\
& m\rq_{2,1}(x,y)-m_{2,1}(x,y)=m_2(\phi_1(x),y)+m_2(x,\phi_1(y))-\phi_1(m_2(x,y))=\delta_2 \phi_1(x,y).\\
\end{align*}
Thus, we have the following proposition.
  \begin{prop}
 Two equivalent deformations have cohomologous infinitesimals.
  \end{prop}
  \begin{proof}
  Suppose $\mathfrak{g}_t=(\mathfrak{g},m_{1,t}, m_{2,t})$ and $\mathfrak{g}'_t=(\mathfrak{g},m'_{1,t}, m'_{2,t})$ be two equivalent one-parameter formal deformations of a compatible Leibniz algebra $\mathfrak{g}$. Suppose $(m_{1,n}, m_{2,n})$ and $(m'_{1,n}, m'_{2,n})$ be two $n$-infinitesimals of the deformations $(m_{1,t}, m_{2,t})$ and $(m'_{1,t}, m'_{2,t})$ respectively. Using Equation (\ref{equivalent 10}) we get,
  \begin{align*}
  &m\rq_{1,n}(x,y)+\phi_n(m\rq_{1,0}(x,y))=m_{1,n}(x,y)+m_{1,0}(\phi_n(x),y)+m_{1,0}(x,\phi_n(y)),\\
  &m\rq_{1,n}(x,y)-m_{1,n}(x,y)=m_{1}(\phi_n(x),y)+m_1(x,\phi_n(y))-\phi_n(m\rq_1(x,y))=\delta_1 \phi_n(x,y),
  \end{align*}
and  
   \begin{align*}
  &m\rq_{2,n}(x,y)+\phi_n(m\rq_{2,0}(x,y))=m_{2,n}(x,y)+m_{2,0}(\phi_n(x),y)+m_{2,0}(x,\phi_n(y)),\\
  &m\rq_{2,n}(x,y)-m_{2,n}(x,y)=m_{2}(\phi_n(x),y)+m_2(x,\phi_n(y))-\phi_n(m\rq_2(x,y))=\delta_2 \phi_n(x,y). 
  \end{align*}
  
  Thus, infinitesimals of two deformations determines same cohomology class.
    \end{proof}
  \begin{thm}
  A non-trivial deformation of a compatible Leibniz algebra is equivalent to a deformation whose infinitesimal is not a coboundary.
  \end{thm}
  \begin{proof}
 Let $(m_{1,t}, m_{2,t})$ be a deformation of the compatible Leibniz algebra $\mathfrak{g}$ and $(m_{1,n}, m_{2,n})$ be the $n$-infinitesimal of the deformation for some $n\geq 1$. Then by Theorem (\ref{infintesimal}), $(m_{1,n}, m_{2,n})$ is a $2$-cocycle, that is, $\delta^2_c (m_{1,n}, m_{2,n})=0$. Suppose $(m_{1,n}, m_{2,n})=-\delta^1_c\phi_n$ for some $\phi_n\in C^1_c(\mathfrak{g}, \mathfrak{g})$, that is, $(m_{1,n}, m_{2,n})$ is a coboundary. We define a formal isomorphism $\Phi_t$ of $\mathfrak{g}[[t]]$ as follows:
  $$\Phi_t(x)=x+\phi_n(x)t^n.$$
  We set
  \begin{align*}
  &\bar{m_{1,t}}=\Phi^{-1}_t\circ m_{1,t}\circ (\Phi_t\otimes\Phi_t),\\
   &\bar{m_{2,t}}=\Phi^{-1}_t\circ m_{2,t}\circ (\Phi_t\otimes\Phi_t).
  \end{align*}
  Thus, we have a new deformation $(\bar{m_{1,t}}, \bar{m_{2,t}})$ which is isomorphic to $(m_{1,t}, m_{2,t})$. By expanding the above equations and comparing coefficients of $t^n$, we get
  \begin{align*}
 & \bar{m_{1,n}}-m_{1,n}=\delta^1 \phi_n,\\
  & \bar{m_{2,n}}-m_{2,n}=\delta_2 \phi_n.
   \end{align*}
  Hence, $\bar{m_{1,n}}=0, ~\bar{m_{2,n}}=0$. By repeating this argument, we can kill off any infinitesimal which is a coboundary. Thus, the process must stop if the deformation is non-trivial. 
  \end{proof}
 \begin{cor}
 Let $(\mathfrak{g}, m_1, m_2)$ be a compatible Leibniz algebra.  If $H^{2}_{com} (\mathfrak{g}, \mathfrak{g})=0$ then $\mathfrak{g}$ is rigid.
 \end{cor}
\subsection{Obstruction and deformation cohomology}
In this subsection, we discuss extensibility and rigidity of  deformations of compatible Leibniz algebras.
\begin{defn}
A deformation of order $n$ of a compatible Leibniz algebra $\mathfrak{g}$ consists of  $\mathbb{K}[[t]]$-bilinear maps $m_{1,t}: \mathfrak{g}[[t]]\times \mathfrak{g}[[t]]\to \mathfrak{g}[[t]]$, $m_{2,t}: \mathfrak{g}[[t]]\times \mathfrak{g}[[t]]\to \mathfrak{g}[[t]]$ of the forms
$$m_{1,t}=\sum^n_{i=0}m_{1,i}t^i,~~~ m_{1,t}=\sum^n_{i=0}m_{1,i}t^i,$$
such that $(m_{1,t}, m_{2,t})$ satisfy all the conditions of a one-parameter formal deformation in the Definition \ref{deform defn} $(mod~ t^{n+1})$.
\end{defn}
A deformation of order $1$ is called an infinitesimal deformation. We say a deformation $(m_{1,t}, m_{2,t})$ of order $n$ of a compatible Leibniz algebra is extendable to a deformation of order $(n+1)$ if there exist elements $m_{1, n+1}, m_{2, n+1}\in C^2_c(\mathfrak{g}, \mathfrak{g})$  such that
\begin{align*}
&\bar{m_{1,t}}=m_{1,t}+m_{1, n+1}t^{n+1},\\
&\bar{m_{2,t}}=m_{2,t}+m_{2, n+1}t^{n+1},
\end{align*}
and $(\bar{m_{1,t}}, \bar{m_{2,t}})$ satisfies all the conditions of Definition \ref{deform defn} $(mod~ t^{n+2})$.

The deformation $(\bar{m_{1,t}}, \bar{m_{2,t}})$ of order $(n+1)$ gives us the following equations.
\begin{align}
&\sum_{i+j=n+1} \big(m_{1,i}(x, (m_{1,j}(y, z))- m_{1,i}(m_{1,j}(x, y),z)+ m_{1,i}(m_{1,j}(x, z),y)\big)=0. \label{obs ass-1-1}\\
&\sum_{i+j=n+1} \big(m_{2,i}(x, (m_{2,j}(y, z)) - m_{2,i}(m_{2,j}(x, y),z) + m_{2,i}(m_{2,j}(x, z),y)\big)=0.\label{obs ass-2-1}\\
&\sum_{i+j=n+1}\big( m_{2,i}(x, m_{1 ,j}(y,z)) + m_{1,i}(x, m_{2 ,j}(y,z))\big) \label{obs com}\\
\nonumber &= \sum_{i+j=n+1}\big( m_{2,i}(m_{1,j}(x, y),z) + m_{1,i}(m_{2,j}(x, y),z)- m_{2,i}(m_{1,j}(x, z),y)- m_{1,i}(m_{2,j}(x, z),y)\big)
\end{align}
 This is same as the following equations
 
 \begin{align}
&\sum_{i+j=n+1} [m_{1,i}, m_{1,j}] = 0,\\
&\sum_{i+j=n+1} [m_{2,i}, m_{2,j}] = 0,\\
&\sum_{i+j=n+1} [m_{1,i}, m_{2,j}] = 0.
\end{align}
Equivalently, we can rewrite the above equations as follows:
 \begin{align}
&\delta_{1}(m_{1,n+1}) =\frac{1}{2}\sum_{\substack{i+j=n+1\\i,j>0}} [m_{1,i}, m_{1,j}],
\end{align}
\begin{align}
&\delta_{ 2}(m_{2,n+1})=\frac{1}{2}\sum_{\substack{i+j=n+1\\i,j>0}} [m_{2,i}, m_{2,j}],
\end{align}
\begin{align}
& \delta_{ 2}(m_{1,n+1}) + \delta_{1}(m_{2,n+1})=\sum_{\substack{i+j=n+1\\i,j>0}} [m_{1,i}, m_{2,j}].
\end{align}
We define the $n$th obstruction to extend a deformation of a Hom-Leibniz algebra of order $n$ to a deformation of order $n+1$  as  $\text{Obs}^n = (\text{O}^n_{m_1}, \text{O}^n_{m_1, m_2}, \text{Obs}^n_{m_2})$, where
\begin{align}
\label{obs equ 222}&\text{O}^n_{m_1} =\frac{1}{2}\sum_{\substack{i+j=n+1\\i,j>0}} [m_{1,i}, m_{1,j}],\\
&\text{O}^n_{m_2} :=\frac{1}{2}\sum_{\substack{i+j=n+1\\i,j>0}} [m_{2,i}, m_{2,j}],\\
&\text{O}^n_{m_1, m_2} :=\sum_{\substack{i+j=n+1\\i,j>0}} [m_{1,i}, m_{2,j}].
\end{align}
Thus,  $\text{O}^n=(\text{O}^n_{m_1}, \text{O}^n_{m_1, m_2}, \text{O}^n_{m_2}) \in C^{3}_{com}(\mathfrak{g}, \mathfrak{g})$ and $(\text{O}^n_{m_1}, \text{O}^n_{m_1, m_2}, \text{O}^n_{m_2}) = \delta^{2}_{com}(m_{1,n+1},m_{2,n+1})$.
\begin{thm}\label{obstruc-thm}
A deformation of order $n$ extends to a deformation of order $n+1$ if and only if the cohomology class of $\text{O}^n$ vanishes.
\end{thm}
\begin{proof}
Suppose a deformation $(m_{1,t},m_{2,t})$ of order $n$ extends to a deformation of order $n+1$. From the obstruction equations, we have
$$\text{O}^n = (\text{O}^n_{m_1},  \text{O}^n_{m_1, m_2}, \text{O}^n_{m_2}) = \delta^{2}_{c}(m_{1,n+1},m_{2,n+1}).$$
As $\delta_c \circ \delta_c=0$, we get the cohomology class of  $\text{O}^n$ vanishes.

Conversely, suppose the cohomology class of  $\text{O}^n$ vanishes, that is,
$$\text{O}^n=\delta^{2}_{c} (m_{1,n+1},m_{2,n+1}),$$
for some $2$-cochains $(m_{1, n+1}, m_{2,n+1})$. We define $(m'_{1,t},m'_{2,t})$ extending the deformation $(m_{1,t},m_{2,t})$ of order $n$ as follows: 
\begin{align*}
&m'_{1,t}=m_{1,t}+m_{1,n+1}t^{n+1},\\
&m'_{2,t}=m_{2,t}+m_{2,n+1}t^{n+1}.
\end{align*}
It is a routine work to check that  $(m'_{1,t}, m'_{2,t})$ defines a formal deformation of order $n+1$.
 Thus, $(m'_{1,t}, m'_{2,t})$ is a deformation of order $n+1$ which extends the deformation $(m_{1,t}, m_{2,t})$ of order $n$.
\end{proof}
\begin{cor}\label{obstruc-cor}
If $H^{3}_{com} (\mathfrak{g}, \mathfrak{g})=0$, then any infinitesimal deformation extends to a one-parameter formal deformation of $(\mathfrak{g},m_1,m_2)$.
\end{cor}
\section{Abelian extensions and cohomology}\label{sec8}
In this section, we show that the second cohomology group $H^{2}_{com} (\mathfrak{g}, M)$ of a compatible Leibniz algebra $(\mathfrak{g},m_1,m_2)$ with coefficients in a compatible bimodule $(M, l_1,r_1, l_2, r_2)$ can be interpreted as equivalence classes of abelian extensions of $\mathfrak{g}$ by $M$.

Let $\mathfrak{g} = (\mathfrak{g}, m_1,m_2)$ be a compatible Leibniz algebra and $M$ be a vector space. Note that $M$ can be considered as a compatible Leibniz algebra with trivial multiplications.

\begin{defn}
An abelian extension of $\mathfrak{g}$ by $M$ is an exact sequence of compatible Leibniz algebras
\[
\xymatrix{
0 \ar[r] &  (M, 0, 0) \ar[r]^{i} & (E, m^E_{1}, mu^E_{2}) \ar[r]^{j} & (\mathfrak{g}, m_1,m_2) \ar[r] \ar@<+4pt>[l]^{s} & 0
}
\]
together with a $\mathbb{K}$-splitting $s$.
\end{defn}

An abelian extension induces a compatible $\mathfrak{g}$-bimodule structure on $M$ via the action map
\begin{align*}
\begin{cases}
&l_1(x, m) = m^E_1 (s(x), i(m))\\
&r_1(m,x)= m^E_1 ( i(m), s(x))
\end{cases}
;
\begin{cases}
&l_2(x, m) = m^E_2 (s(x), i(m))\\
&r_2(m,x)= m^E_2 ( i(m), s(x)).
\end{cases}
\end{align*}

One can easily verify that this action is independent from the choice of $s$. 


Two abelian extensions are said to be equivalent if there is a map $\phi : E \rightarrow E'$ between compatible Leibniz algebras making the following diagram commute
\[
\xymatrix{
0 \ar[r] &  (M, 0,0) \ar[r]^{i} \ar@{=}[d] & (E, m^E_1,m^E_2) \ar[d]^{\phi} \ar[r]^{j} & (\mathfrak{g}, m_1,m_2) \ar[r] \ar@{=}[d] \ar@<+4pt>[l]^{s} & 0 \\
0 \ar[r] &  (M, 0, 0) \ar[r]^{i'} & (E', m'^{E}_1, m'^{E}_2) \ar[r]^{j'} & (\mathfrak{g}, m_1,m_2) \ar[r] \ar@<+4pt>[l]^{s'} & 0 .
}
\]
Observe that two extensions with same $i$ and $j$ but different $s$ are always equivalent.

Suppose $M$ is a given $\mathfrak{g}$-bimodule. We denote by $\mathcal{E}xt _{com}(\mathfrak{g}, M)$ the equivalence classes of abelian extensions of $\mathfrak{g}$ by $M$ for which the induced $\mathfrak{g}$-bimodule structure on $M$ is the prescribed one.

The next result is inspired from the classical case.
\begin{thm}\label{thm-abelian-ext}
$H^{2}_{com} (\mathfrak{g}, M) \cong \mathcal{E}xt_{com} (\mathfrak{g}, M).$
\end{thm}

\begin{proof}
Given a $2$-cocycle $f \in C^{2}_{com} (\mathfrak{g}, M)$, we consider the $\mathbb{K}$-module $E = M \oplus \mathfrak{g}$ with following structure maps
\begin{align*}
{\mu}^E_1 ((m, x), (n, y)) =~& ( r_1(m , y) + l_1( x , n) + f (x, y),~ m_1 (x, y)),\\
{\mu}^E_2 ((m, x), (n, y)) =~& ( r_2(m , y) + l_2( x , n) + f (x, y),~ m_2 (x, y)).
\end{align*}
(Observe that when $f =0$ this is the semi-direct product).
Using the fact that $f$ is a $2$-cocycle, it is easy to verify that $(E, \mu^E_1, \mu^E_2)$ is a compatible Leibniz algebra. Moreover, $0 \rightarrow M \rightarrow E \rightarrow \mathfrak{g} \rightarrow 0$ defines an abelian extension with the obvious splitting. Let $(E' = M \oplus \mathfrak{g}, \mu'^{E}_1, \mu'^{E}_2)$ be the corresponding compatible Leibniz algebra associated to the cohomologous $2$-cocycle $f - \delta^{1}_{com} (g)$, for some $g \in C^{1}_{com} (\mathfrak{g}, M)$. The equivalence between abelian extensions $E$ and $E'$ is given by $E \rightarrow E'$, $(m, x) \mapsto (m + g (x), x)$. Therefore, the map $H^2_{com} (\mathfrak{g}, M) \rightarrow \mathcal{E}xt_{com} (\mathfrak{g}, M) $ is well-defined.

Conversely, given an extension 
$0 \rightarrow M \xrightarrow{i} E \xrightarrow{j} \mathfrak{g} \rightarrow 0$ with splitting $s$, we may consider $E = M \oplus \mathfrak{g}$ and $s$ is the map $s (x) = (0, x).$ With respect to the above splitting, the maps $i$ and $j$ are the obvious ones. 
Since $j \circ m^E_1 ((0, x), (0, y)) = m_1 (x, y)$, and $j \circ m^E_2 ((0, x), (0, y)) = m_2 (x, y)$ as $j$ is an algebra map, we have $m^E_1 ((0, x), (0, y)) = (f (x, y), m_1 (x, y))$, and $m^E_2 ((0, x), (0, y)) = (f (x, y), m_2 (x, y))$,  for some $f \in C^{2}_{com} (\mathfrak{g}, M).$ The Leibniz condition of $m^E_1, m^E_2$ then implies that $f$ is a $2$-cocycle. Similarly, one can observe that any two equivalent extensions are related by a map $E = M \oplus \mathfrak{g} \xrightarrow{\phi} M \oplus \mathfrak{g} = E'$, $(m, x) \mapsto (m + g(x), x)$ for some $g \in C^{1}_{com} (\mathfrak{g}, M)$. Since $\phi$ is an algebra morphism, we have
\begin{align*}
&\phi \circ m^E_1 ((0, x), (0, y)) = m'^{E}_1 (\phi (0, x) , \phi (0, y)),\\
&\phi \circ m^E_2 ((0, x), (0, y)) = m'^{E}_2 (\phi (0, x) , \phi (0, y)),
\end{align*}
which implies that $f' (x, y) = f (x, y) - (\delta_{com} ~g)(a, b)$. Here $f'$ is the $2$-cocycle induced from the extension $E'$. This shows that the map $\mathcal{E}xt_{com} (\mathfrak{g}, M) \rightarrow H^{2}_{com} (\mathfrak{g}, M)$ is well-defined. Moreover, these two maps are inverses to each other.
\end{proof}

{\bf Acknowledgements:} . The first author is supported   by
IFCPAR/CEFIPRA  (Grant No. 6201-C
2019-0071. The second author is supported by the Science and Engineering Research Board (SERB), Department of Science and Technology (DST), Govt. of India. (Grant Number- CRG/2022/005332)

\end{document}